\documentclass[11pt]{article}

\usepackage{amsmath,amssymb,amsthm,graphics}

\newtheorem{theorem}{Theorem}[section]
\newtheorem{lemma}[theorem]{Lemma}

\theoremstyle{definition}

\newtheorem{corollary}[theorem]{Corollary}

\newtheorem{prop}[theorem]{Proposition}

\theoremstyle{remark}

\numberwithin{equation}{section}

\newtheorem*{T1}{Theorem 1}

\newtheorem*{C2}{Corollary 2}

\theoremstyle{remark}

\newcommand{\Ext}{\mathrm{Ext}}
\newcommand{\Hom}{\mathrm{Hom}}

\title{The second cohomology of simple $SL_2$-modules}
\author{David I. Stewart\\Department of Mathematics, Imperial College, London, SW7 2AZ, UK\\davis.stewart06@imperial.ac.uk}
\date{April 9, 2009 and, in revised form, June 29, 2009.}

\begin{document}
\maketitle
\begin{abstract}Let $G$ be the simple algebraic group $SL_2$ defined over an algebraically closed field $K$ of characteristic $p>0$. In this paper, we compute the second cohomology of all irreducible representations of $G$.\footnote[1]{This paper was prepared towards the author's PhD qualification under the supervision of Prof. M. W. Liebeck, with financial support from the EPSRC. We would like to thank Prof. Liebeck for his help in producing this paper. Additional thanks are due to the anonymous referee who made very helpful suggestions for improvements to the paper.}\end{abstract}
\section{Introduction}
Let $G=SL_2$ defined over an algebraically closed field $K$ of characteristic $p>0$. Fix a maximal torus $T$, contained in a Borel subgroup $B$ of $G$. Recall that the weight lattice $X(T)$ of $T$ can be identified with $\mathbb Z$, while the choice dominant weights of $T$ associated with the choice of $B$ can be identified with $\mathbb Z^+$. Thus for each positive integer $r$ there is an irreducible module $L(r)$ of highest weight $r$ . Let $V$ be a $G$-module. As $G$ is defined over $\mathbb F_p$ we have the notion of the $d$th Frobenius map which raises each matrix entry in $G$ to the power $p^d$. When composed with the representation $G\to GL(V)$, this induces the twist $V^{[d]}$ of $V$.

\begin{T1} Let $V=L(r)^{[d]}$ be any Frobenius twist (possibly trivial) of the irreducible $G$-module $L(r)$ with highest weight $r$ where $r$ is one of \begin{align*}&2p\\
&2p^2-2p-2\ (p>2) \\
&2p-2+(2p-2)p^{e}\ (e>1)\end{align*}
Then $H^2(G,V)=K$. For all other irreducible $G$-modules $V$, $H^2(G,V)=0$.
 \end{T1}

But perhaps more strikingly, combining Theorem 1 and Proposition \ref{ext1} below yields
\begin{C2}Let $V=L(r)^{[d]}$ be an irreducible $G$-module. Assume that either $p>2$ or $r\neq 2p$. Then $H^2(G,V)\neq 0$ only if there exists some irreducible module $W$ with $H^1(G,W)\neq 0$ and $\Ext^1(W,V)\neq 0.$
\end{C2}

Notice too that the analogous result holds for $H^1(G,V)$; that is, $H^1(G,V)\neq 0$ for an irreducible $G$-module $V$ if and only if $\Ext^1(W,V)\neq 0$ for some irreducible module $W$ with $H^0(G,W)\neq 0$; that is, $H^1(G,V)\neq 0$ if and only if $\Ext^1(K,V)\neq 0$.

After the submission of this paper, we were made aware of the existence of a more general result due to A. Parker \cite{park}, which gives a recursive formula for all extensions between Weyl modules. This in turn can be used with further recursive formulas given in that paper to generate the above result, in principle. However, this would require some work. Given the usefulness of our result to group theorists, and the simplicity of this proof, we feel that this paper still makes a valuable contribution.

This paper was inspired by the methods of G. McNinch's paper \cite{mcninch} which computes $H^2(G,V)$ for simply connected algebraic groups $G$ acting on modules $V$ with $\dim V\leq p$. The main result of this paper makes a clarification to Theorem A of \cite{mcninch} in the case  $G=SL_2(K)$ and $p=2$. Namely if $V=L(4)=L(1)^{[2]}$ is the second Frobenius twist of the natural representation of $G$ then $\dim V=2\leq 2=p$ and $H^2(G,W)=K\neq 0$ for any Frobenius twist $W=V^{[d]}$ of $V$ $(d\geq 0)$.

The remainder of the paper is dedicated to proving Theorem 1. Along the way we derive the $\Ext^1$  result \ref{ext1} originally due to Cline and rederived in \cite{AJL}.

\section{Proof of Theorem 1}
We begin with a little notation which we keep compatible with \cite{J}:

Let $B$ be a Borel subgroup of a reductive algebraic group $G$, containing a maximal torus $T$ of $G$. Following \cite{J} we choose that Lie$(B)$ contains only negative root spaces of Lie$(G)$. Recall that for each dominant weight $\lambda\in X(T)$ for $G$, the space $H^0(\lambda):=H^0(G/B,\lambda)=\mathrm{Ind}_B^G(\lambda)$ is a $G$-module with highest weight $\lambda$ and with socle $\mathrm{Soc}_G H^0(\lambda)=L(\lambda)$, the irreducible $G$-module of highest weight $\lambda$. We have the Weyl module of highest weight $\lambda$, $V(\lambda)\cong H^0(-w_0\lambda)^*$ where $w_0$ is the longest element in the Weyl group. For $G=SL_2$, we identify $X(T)$ with $\mathbb Z$. In this case  $H^0(i)=V(i)^*$. When $0\leq i<p$, we say that $i$ is a restricted weight and $H^0(i)=L(i)$ the irreducible $G$-module of high weight $i$. Steinberg's tensor product theorem gives a description of all irreducible modules in terms of Frobenius twists of the restricted ones: let $r=r_0+pr_1+p^2r_2+\dots+p^nr_n$ be the $p$-adic expansion of the integer $r$; then $L(r)\cong L(r_0)\otimes L(r_1)^{[1]}\otimes \dots\otimes L(r_n)^{[n]}$. From now on we will drop the $L$ and refer simply to $r$ as the irreducible $G$-module of high weight $r$; we write $L(0)=K$ to avoid confusion. We will often need to refer to specific parts of this tensor product, so we write $r=r_0\otimes r_1^{[1]}\otimes\dots\otimes r_n^{[n]}=r_0\otimes r'^{[1]}=r_0\otimes r_1^{[1]}\otimes r''^{[2]}$, where $r'=r_1\otimes r_2^{[1]}\otimes \dots\otimes r_n^{[n-1]}$ and $r''=r_2\otimes r_3^{[1]}\otimes\dots\otimes r_n^{[n-2]}$.

Recall the dot action of the (affine) Weyl group on the weight lattice $X(T)$: $w\centerdot \lambda=w(\lambda+\rho)-\rho$, where $\rho$ is the half-sum of the positive roots. In case $G=SL_2$, $\rho=1$. Denote by $G_1$ the first Frobenius kernel of $G$; it is an infinitesimal, normal subgroup scheme of $G$ (see for instance, \cite[I.9]{J}). We repeatedly use the linkage principle for $G$ and $G_1$ (see \cite[II.6.17]{J} and \cite[II.9.16]{J}); in the case $G=SL_2$, it is easy to check that this means  that if $\Ext_G^i(r,s)\neq 0$ or $\Ext_{G_1}^i(r,s)\neq 0$ for any $i\geq 0$ then we must have $r_0=s_0$ or $r_0=p-2-s_0$ (for $p=2$ the latter means that $r_0=s_0=0$).

Let $V$ be a $G$-module. As $G_1\triangleleft G$, the cohomology group $M=H^i(G_1,V)$ has the structure of a $G/G_1$-module, and so also of a $G$-module. Since $G_1$ acts trivially on this module, there is a Frobenius untwist $M^{[-1]}$ of $M$.  By \cite[I.9.5]{J}, $G/G_1\cong F_1(G)$, where $F_1$ is the first Frobenius morphism. Thus $G/G_1$ acts on $H^i(G_1,V)$ as $G$ acts on $H^i(G_1,V)^{[-1]}$.

There are two main ingredients in the proof of Theorem 1. The first is the Lyndon-Hochschild-Serre spectral sequence \cite[6.6 (3)]{J} applied to $G_1\triangleleft G$, using the observations in the preceding paragraph.

\begin{prop}\label{lhs} There is for each $G$-module $V$ a spectral sequence \[E_2^{nm}=H^n(G,H^m(G_1,V)^{[-1]})\Rightarrow H^{n+m}(G,V).\]\end{prop}

We will always refer by $E_*^{**}$ to terms in the above spectral sequence. We briefly recall the important features of spectral sequences for the unfamiliar reader. The lower subscript refers to the sheet of the spectral sequence. Only the second sheet is explicitly defined in this example. The point $E_2^{nm}$ is defined only for the first quadrant, i.e. for $n,m\geq 0$; for any other $n,m$ we have $E_2^{nm}=0$. On the $i$th sheet, maps in the spectral sequence through the $n,m$th point go \[\dots\to E_i^{n-i,m+i-1}\stackrel{\rho}{\to} E_i^{nm}\stackrel{\sigma}{\to}E_i^{n+i,m-i+1}\to\dots.\] These form a complex, i.e. $\sigma\rho=0$. One then gets the point of the next sheet, $E_{i+1}^{nm}$ as a section of $E_i^{nm}$ by ker $\sigma/$im $\rho$, i.e. the cohomology at that point. If $i$ is big enough, maps go from outside the defined quadrant to a given point and then from that point to outside the defined quadrant again. Thus each point of the spectral sequence evenually stabilises. We denote the stable value by $E_\infty^{nm}$. Finally, one gets \begin{equation}\label{sumup}H^r(G,V)=\bigoplus_{n+m=r}E_\infty^{nm}\end{equation} explaining the notation '$\Rightarrow H^{n+m}(G,V)$'.

The second main ingredient is the following proposition, which is a specialisation of the main theorem from \cite{AJ} to the case $G=SL_2$.

\begin{prop}\label{cfk}Let  $0\leq r<p$ denote a restricted high weight for $G$ in integral notation, where we write $L(0)=K$. Then as $G$-modules, \begin{align*}(i) && H^{2i}(G_1,K)^{[-1]}&\cong H^0(2i),\\
(ii) && H^{2i+1}(G_1,p-2)^{[-1]}&\cong H^0(2i+1),\\
(iii) && H^{i}(G_1,r)^{[-1]}&=0\text{, if }r\neq 0, p-2\end{align*}
\end{prop}
\begin{proof} We use the main result of \cite{AJ}: Let $w\in W$ be an element of length $l(w)$ of the Weyl group $W=N_G(T)/T$ of the reductive algebraic group $G$ with Borel subgroup $B$ and maximal torus $T\leq B$. Let $\mathfrak{u}$ denote the Lie algebra of the unipotent radical $U$ of $B$. Then
{\footnotesize \[H^i(G_1,H^0(G/B,w\centerdot 0+p\lambda))^{[-1]}\cong\begin{cases}H^0(G/B,S^{(i-l(w))/2}(\mathfrak{u}^*)\otimes\lambda)\\\hspace{20pt}\text{if }(i-l(w))\text{ is even and } w\centerdot 0+p\lambda\text{ is dominant,}\\0\text{ otherwise}.\end{cases}\]}
where $S^i(\mathfrak{u}^*)$ is the $i$th symmetric power. This theorem is proved under certain restrictions, but it is always true for $G=SL_2$. The result for $p>2$ is \cite[3.7]{AJ} and $p=2$ is \cite[3.10 (3)]{AJ}.

We apply this to $G=SL_2$.  Here, $w$ is either $1$, the identity element, or $s$, the involution, with $l(w)=0$ or $1$ respectively, and with $w\centerdot 0=0$ or $-2$ respectively. Thus the only restricted dominant weights of the form $w\centerdot 0+p\lambda$ are $0$ or $p-2$, giving $(iii)$ by the linkage principle for $G$.  So assume $w\centerdot 0+p\lambda$ is $0$ or $p-2$ and $H^i(G_1,H^0(w\centerdot 0+p\lambda))^{[-1]}\neq 0$. Observe that $\mathfrak u^*$ is a one-dimensional $B$-module on which $T$ acts with high weight $2$, and on which $U$ acts trivially. Thus $\mathfrak u^*$ can be regarded as a character for $T$ of weight $2$. So $S^i(\mathfrak u^*)$ can be regarded as a character for $T$ of high weight $2i$. If $i=2i'$, then $l(w)=0$, $w=1$, and we must take $\lambda=0$ giving $S^{i-l(w)}(\mathfrak u^*)\otimes \lambda=S^{i'}(\mathfrak u^*)$, a character of weight $2i'$. If $i=2i'+1$, then $l(w)=1$, $w=s$, and we must take $\lambda=1$ giving $S^{i-l(w)}(\mathfrak u^*)\otimes \lambda=S^{i'}(\mathfrak u^*)\otimes 1$, a character of weight $2i'+1$.
\end{proof}

From now on, let $G=SL_2$. Let $r=r_0\otimes r'^{[1]}$ and $s=s_0\otimes s'^{[1]}$ denote two irreducible $G$-modules. We first apply the spectral sequence to the module $V=r\otimes s$ to rederive Cline's result on $\Ext_G^1(r,s)$. We do this because it is a straightforward application of our method, and because we need it for the proofs of Theorem 1 and Corollary 2.

Notice that $H^m(G_1,r\otimes s)=H^m(G_1,(r_0\otimes s_0)\otimes r'^{[1]}\otimes s'^{[1]})$ as $G$-modules: it is enough to check this when $m=0$, i.e. $\Hom_{G_1}(K,r\otimes s)=\Hom_{G_1}(K,r_0\otimes s_0)\otimes r'^{[1]}\otimes s'^{[1]}$, which is a consequence of Steinberg's tensor product theorem. Now  \[E_2^{nm}=H^n(G,H^m(G_1,r\otimes s)^{[-1]})=H^n(G,H^m(G_1,r_0\otimes s_0)^{[-1]}\otimes r'\otimes s').\]
First we require the following:

\begin{lemma}\label{steinberg} \[\Ext^i_{G_1}(p-1,p-1)=\begin{cases}K\text{ if }i=0\\0\text{ otherwise.}\end{cases}\]\end{lemma}
\begin{proof} Observe that $p-1$ represents the Steinberg module $St_1$ for $G$ and hence for $G_1$. The lemma now follows as $St_1$ is an injective (hence projective) module for $G_1$. See \cite[II.3.18(4,5)]{J} and \cite[II.11.1 Remark (2)]{J}.\end{proof}

\begin{prop}\label{23page}In the spectral sequence of Proposition \ref{lhs} applied to $V=r\otimes s$, 
\begin{enumerate}
\item if $m$ is even and $E_2^{nm}\neq 0$ then $r_0=s_0$ and $E_2^{n'm'}=0$ for any $n',m'$ with $m'$ odd;
\item if $m$ is odd and $E_2^{nm}\neq 0$ then $s_0=p-2-r_0$ and $E_2^{n'm'}=0$ for any $n',m'$ with $m'$ even.
\item Further when $p=2$, if $m$ is even and $E_2^{nm}\neq 0$ then $r_1=s_1$, whereas if $m$ is odd and $E_2^{nm}\neq 0$ then $r_0=s_0=0$ and $r_1=s_1\pm 1$.\end{enumerate}\end{prop}
\begin{proof}Assume $E_2^{nm}\neq 0$. Thus $H^m(G_1,r_0\otimes s_0)=\Ext_{G_1}^m(r_0,s_0)\neq 0$ so we have $r_0=s_0$ or $p-2-r_0=s_0$ by the linkage principle for $G_1$.

Firstly, let $p>2$. If $s_0=p-2-r_0$, then all the weights in $r_0\otimes s_0$ are odd and less than $p$. Thus there is no $G_1$-composition factor which is a trivial module. So if $m$ is even and $E_2^{nm}\neq 0$ then by \ref{cfk}, $r_0=s_0$.

On the other hand, If $m$ is odd and $E_2^{nm}\neq 0$ then $r_0\otimes s_0$ must have a $G_1$-composition factor which is $p-2$ by \ref{cfk}. Now if $r_0=s_0$ the weights of $r_0\otimes s_0$ are all even, so if $r_0\otimes s_0$ has a $G_1$-composition factor which is $p-2$, it must have a $G$-composition factor which is $p-2\otimes 1^{[1]}$. Thus we must have $r_0=p-1$. But then $m=0$ by \ref{steinberg}; a contradiction as $m$ is odd. So we must have $s_0=p-2-r_0$. This proves (i) and (ii).

When $p=2$, observe that each $r_i,s_i$ takes the value $0$ or $1$. Assume $E_2^{nm}\neq 0$. Then $H^m(G_1,r_0\otimes s_0)=\Ext_{G_1}^m(r_0,s_0)\neq 0$ and so $r_0=s_0$ by the linkage principle for $G_1$. If $r_0=s_0=1$, then $m=0$ by \ref{steinberg} and $H^m(G_1,r_0\otimes s_0)^{[-1]}=H^0(0)=K$. On the other hand, when $r_0=s_0=0$, $H^m(G_1,r_0\otimes s_0)^{[-1]}=H^0(m)$ by \ref{cfk}. Since $E_2^{nm}=H^n(G,H^0(m)\otimes r'\otimes s')\neq 0$, we require the weights of $H^0(m)\otimes r'\otimes s'$ to be even by the linkage principle. Thus if $m$ is even, $r_1=s_1$ and if $m$ is odd, $r_1=s_1\pm 1$.\end{proof}

\begin{corollary}\label{cor} In the spectral sequence $E_i^{nm}$ applied to $V=r\otimes s$, we have
\begin{align*}E_2^{nm}&=E_3^{nm}\text{ for any } n,m\geq 0\\
E_2^{01}&=E_\infty^{01}\\
H^1(G,V)&=E_2^{10}\oplus E_2^{01}\\
E_2^{20}&=E_\infty^{20}\\
E_2^{11}&=E_\infty^{11}.\end{align*}
\end{corollary}
\begin{proof}
The first statement is clear in light of Proposition \ref{23page} since $E_3^{nm}$ is the cohomology of $E_2^{n-2,m+1}\to E_2^{nm}\to E_2^{n+2,m-1}$ and if the middle term is non-zero, then the outer terms must be zero. The second statement then follows from (\ref{sumup}) above, noting the obvious identity  $E_2^{10}=E_\infty^{10}$.

The remaining statements are now clear. For example the last statement follows since $E_2^{11}=E_3^{11}$ and the maps in $E_3$ are $E_3^{n-3,m+2}\to E_3^{nm}\to E_3^{n+3,m-2}$ so $E_4^{11}$ is the cohomology of $0\to E_3^{11}\to 0$.\end{proof}

Now we are in a position to rederive Cline's result on $\Ext_G^1$ as stated in \cite[3.9]{AJL}. In the statement below, note that by $s_{k+1}\pm 1$ or $p-2-s_k$ we only refer to those values lying between $0$ and $p-1$ inclusive.

\begin{prop}[Cline]\label{ext1} Let $r=r_0+pr_1+\dots+p^nr_n$ and $s=s_0+ps_1+\dots+p^m s_m$ be the $p$-adic expansions of $r$ and $s$. Then $\Ext^1_G(r,s)=K$ if there is some $k$ such that $r_i=s_i$ for $i\neq k,k+1$ and $r_k=p-2-s_k$, $r_{k+1}=s_{k+1}\pm 1$. $\Ext^1_G(r,s)= 0$ otherwise.
\end{prop}
\begin{proof}We remind the reader of the notation $r=r_0\otimes r'^{[1]}$ and $r=r_0\otimes r_1^{[1]}\otimes r''^{[2]}$.

We have that $\Ext^1_G(r,s)=H^1(G,r\otimes s)=E_2^{01}\oplus E_2^{10}$ from the corollary above. Assume $\Ext^1_G(r,s)\neq 0$. By \ref{23page} exactly one of $E_2^{01}$ and $E_2^{10}$ is non-zero. Note that $r\neq s$ by \cite[II.2.12 (1)]{J}. Thus for $p>2$ define $k$ to be the least integer such that $r_k\neq s_k$. When $p=2$ define $k+1$ to be the least integer such that $r_{k+1}\neq s_{k+1}$ (recall $r_0=s_0$ by the linkage principle, so that $k\geq 0$).

Suppose $k>0$, so $r_0=s_0$ and if $p=2$, $r_1=s_1$. Then by \ref{23page}, $E_2^{01}=0$ and so $\Ext_G^1(r,s)=E_2^{10}$. But \[E_2^{10}=H^1(G,H^0(G_1,r_0\otimes s_0)^{[-1]}\otimes r'\otimes s')=\Ext_G^1(r',s'),\] since $H^0(G_1,r_0\otimes s_0)=\Hom_{G_1}(r_0,s_0)=K$ by Schur's lemma as $K$ is algebraically closed and as there is an equivalence of categories of $G_1$-modules and modules over a certain finite dimensional K-algebra. So $\Ext_G^1(r, s)=\Ext_G^1(r', s')$. Continuing to calculate the right hand side inductively, we may assume $k=0$. Now from \ref{23page} again, $r_0=p-2-s_0$ (and if $p=2$, $r_1=s_1\pm 1$),  $E_2^{10}=0$ and so $\Ext_ G^1(r,s)=E_2^{01}$. We now show that $E_2^{01}=K$.

As the highest weight of $r_0\otimes (p-2-r_0)$ is $p-2$, and the $p-2$ weight space is one-dimansional, the module $r_0\otimes (p-2-r_0)$ has a single $p-2$ composition factor. By the linkage principle, no composition factor of $r_0\otimes (p-2-r_0)$ has a non-trivial extension by $p-2$. Thus $p-2$ appears as a direct summand of this tensor product. So by \ref{cfk}
 \[E_2^{01}=H^0(G,1\otimes r'\otimes s')=\Hom_G(r',s'\otimes 1)\] 
and so our assumption requires that  $s'\otimes 1$ contains a simple submodule isomorphic to $r'$. Provided $s_1\neq 0,p-1$,  we have that $s'\otimes 1=(s_1+1)\otimes s''^{[1]}\oplus (s_1-1)\otimes s''^{[1]}$ since $s'\otimes 1$ is self-dual and has exactly two composition factors. So $r'=(s_1\pm 1)\otimes s''^{[1]}$ and $E_2^{01}=K$ as required. If $s_1=0$, then $s'\otimes 1$ is irreducible and again, $r'=(s_1+1)\otimes s''^{[1]}$ with $E_2^{01}=K$ as required. Lastly, if $s_1=p-1$, it is easy to check that $s_1\otimes 1$ has Loewy length 3 with composition factors $p-2$, $p-2$ and $1^{[1]}$ and so either $r_1=p-2$ or $r_1=0$. If $r_1=0$, then $r'\otimes 1$ is irreducible and so $\Hom_G(r',1\otimes s')=\Hom_G(r'\otimes 1,s')=0$ since $s'\neq r'\otimes 1$. So $r_1=p-2$. If $p>2$ then $E_2^{01}=\Hom_G(r'\otimes 1,s')=\Hom_G((p-3\oplus p-1)\otimes r''^{[1]},s')$ giving again $r'=(s_1-1)\otimes s''^{[1]}$ and $E_2^{01}=K$ as required. If $p=2$, recall by the linkage principle that $r_0=s_0$ so that $E_2^{01}=\Hom_G((s_1\pm1)\otimes r''^{[1]},s')$, as observed already, giving $r'=(s_1\pm 1)\otimes s''^{[1]}$ and $E_2^{01}=K$.
\end{proof}

We extend this method to calculate $H^2(G,V)$. We will now use the spectral sequence $E_2^{nm}$ applied to $V=r$. We still have the conclusion of \ref{23page} from the specialisation of \ref{23page} to $s=K$, the trivial module.

\begin{prop}\label{abut} $E_2^{nm}=E_3^{nm}=\dots=E_\infty^{nm}$ for all $n+m<p$.\end{prop}
\begin{proof}Observe that $H^0(m)=L(m)$ is irreducible for all $m$ under consideration. 

Notice that we have already proved the proposition for $p=2$ in \ref{cor}, so assume $p>2$.

Assume inductively that we have the result for $E_2,\dots,E_i$, i.e. $E_2^{nm}=E_3^{nm}=\dots=E_i^{nm}$ for all $n+m<p$. The maps in $E_i$ are $E_i^{n-i,m+i-1}\to E_i^{nm}\to E_i^{n+i,m-i+1}$. We show that \[E_2^{nm}=E_i^{nm}\neq 0\text{ implies }E_2^{n-i,m+i-1}=E_2^{n+i,m-i+1}=0.\hspace{40pt}(*)\] Since $E_i^{n'm'}$ is a section of $E_2^{n'm'}=0$ for any $n,'m'$, $(*)$ will give $E_{i+1}^{nm}$ as the cohomology of the sequence $0\to E_2^{nm}\to 0$, hence $E_{i+1}^{nm}=E_2^{nm}$ as required.

Firstly if $i$ is even, then $(*)$ holds by \ref{23page}, noting also that $E_2^{10}=E_\infty^{10}$ for any such spectral sequence.

We wish to show that for the sequence $E_i^{n-i,m+i-1}\to E_i^{nm}\to E_i^{n+i,m-i+1}$, the second upper index for each non-zero term represents a restricted weight for $G$; i.e. $m+i-1$, $m$, $m-i+1<p$, when the corresponding term is non-zero.

Since $n+m<p$ and $n\geq 0$ we have $m<p$. Also since $i>1$, $m-i+1<p$. Now if $E_i^{n-i,m+i-1}\neq 0$, then $n-i\geq 0$ so if $m+i-1>p-1$, then $n-i+m+i-1>p-1$ so $n+m>p$, which is not true. So $m<p$, $m+i-1<p$ and $m-i+1<p$ as required.  

Now let $i$ be odd and assume $E_2^{nm}\neq 0$. Then $H^m(G_1,r)^{[-1]}$ is non-zero, and so it is equal to $H^0(m)\otimes r'=m\otimes r'$. Thus $E_2^{nm}=H^n(G,H^m(G_1,r)^{[-1]})=\Ext_G^n(m,r')\neq 0$ and so $r'$ is linked to $m$; so $r_1=m$ or $r_1=p-2-m$. If $E_2^{n+i,m-i+1}\neq 0$ then $\Ext_G^{n+i}(m-i+1,r')\neq 0$ and $r'$ is linked to $m-i+1$ as well. But $m-i+1=m-2i'$ where $i':=(i-1)/2\in \mathbb N$ as $i$ is odd. So $r_1=m-2i'$ or $r_1=p-2-m+2i'$. Either of the conditions (a): $r_1=m-2i'$ or (b): $r_1=p-2-m+2i'$ together with either of the conditions (c): $r_1=m$ or (d): $r_1=p-2-m$ produce contradictions: (a) and (d) or (b) and (c) imply that $p$ is even, a contradiction, while (a) and (c) or (b) and (d) imply that $i'>0$, another contradiction. So we cannot have both $E_2^{nm}$ and $E_2^{n+i,m-i+1}$ non-zero. Similarly $E_2^{n-i,m+i-1}$ and $E_2^{nm}$ are not both non-zero. So $(*)$ holds again.\end{proof}

Using (\ref{sumup}), we get

\begin{corollary}\label{hi}For $p>i$, $H^i(G,r)=E_2^{0i}\oplus E_2^{1,i-1}\oplus\dots\oplus E_2^{i0}$.\end{corollary}

\begin{prop}\label{h2}For any $p$,  $H^2(G,r)=E_2^{02}\oplus E_2^{11}\oplus E_2^{02}$.\end{prop}
\begin{proof}The corollary above gives the result for $p>2$. If $p=2$, from \ref{cor}, we have $E_2^{20}=E_\infty^{20}$ and $E_2^{11}=E_\infty^{11}$ so by (\ref{sumup}) we need only prove that $E_2^{02}=E_\infty^{02}.$ We already have $E_2^{02}=E_3^{02}$ from \ref{cor} and so we only need to show that $E_3^{02}=E_4^{02}$ as $E_4^{02}=E_\infty^{02}$. The maps through $E_3^{02}$ in $E_3$ are
\[0\to E_3^{02}\to E_3^{30}.\]
Assume $E_3^{02}$ is non-zero. This implies that \[H^0(G,H^2(G_1,r_0)^{[-1]}\otimes r')\neq 0\]
so $r_0=0$ and $H^0(G,H^0(2)\otimes r')\neq 0$ by \ref{cfk}. But $H^0(2)$ is an indecomposable module $0/1^{[1]}$, i.e. it has socle $1^{[1]}$ and quotient a trivial module. So $\Hom_G(1^{[1]}/0,r')\neq 0$.
This means $r'$ is an irreducible quotient of $1^{[1]}/0$ giving $r'=1^{[1]}$. Now \begin{align*}E_2^{30}&=H^3(G,H^0(G_1,r_0)^{[-1]}\otimes r')\\
 &=H^3(G,1^{[1]})\\&=H^2(G,H^0(2)/\text{soc}_G H^0(2))\text{ by \cite[II.4.14]{J}}\\&=H^2(G,K)\\&=0\text{ by \cite[II.4.11]{J}}.\end{align*}
 Hence $E_4^{02}$ is the cohomology of the sequence $0\to E_3^{02}\to 0$ as required.
\end{proof}

\begin{lemma}\label{lem}For any irreducible $G$-module $r'$, \[H^1(G,1\otimes r')=\begin{cases}K\text{ if }r'= (p-3)\otimes 1^{[1]}\ (p>2)\\
K\text{ if }r'=1\otimes (p-2)^{[e]}\otimes 1^{[e+1]}\ (e>0)\\
0\text{ otherwise.}\end{cases}\]\end{lemma}
\begin{proof} Observe that $H^1(G,1\otimes s)=\Ext_G^1(1,s)$ and so this is an easy check using  \ref{ext1}.\end{proof}

We can now finish the

{\it Proof of Theorem 1:} 

Let $V=r=s^{[d]}$ with $s_0\neq 0$.
We have from \ref{h2} that $H^2(G,V)=E_2^{20}\oplus E_2^{11}\oplus E_2^{02}$. Assume $H^2(G,V)\neq 0$.

Let $p>2$ initially. If $d=0$, then by the linkage principle for $G$, $r_0=p-2$, and $E_2^{02}=E_2^{20}=0$ by \ref{23page}. So $H^2(G,r)=E_2^{11}$ and the latter is $H^1(G,H^1(G_1,p-2)^{[-1]}\otimes r')=H^1(G,1\otimes r')$ using \ref{cfk}. Hence $r'$ is as listed in \ref{lem}, and so $r=(p-2)\otimes r'^{[1]}=(p-2)\otimes (p-3)^{[1]}\otimes 1^{[2]}$ or $(p-2)\otimes 1^{[1]}\otimes (p-2)^{[e+1]}\otimes 1^{[e+2]}$. In these cases, $H^2(G,r)=K$.

If $d>0$ then $r_0=0$ and so $E_2^{11}=0$ by \ref{23page} again. Assume $E_2^{02}\neq 0$. Now  $E_2^{02}=H^0(G,H^0(2)\otimes s'^{[d]})=\Hom_G(2,s^{[d-1]})$. So $d=1$, $s=2$ and $r=2^{[1]}$, giving $E_2^{02}=K$. Now, $E_2^{20}=H^2(G,H^0(G_1,K)^{[-1]}\otimes r')=H^2(G,r')$ again using \ref{cfk}. But this is $0$ if $r'=2$ by the linkage principle for $G$, so $H^2(G,2^{[1]})=E_2^{02}=K$. If $r\neq 2^{[1]}$ then $E_2^{02}=0$ and so $H^2(G,s^{[d]})=E^{20}_2=H^2(G,s^{[d-1]})$. Hence arguing by induction on $d$, we have shown that if $H^2(G,V)\neq 0$ and $p>2$ then $V$ is some Frobenius twist of $2^{[1]}$, $(p-2)\otimes (p-3)^{[1]}\otimes 1^{[2]}$ or $(p-2)\otimes 1^{[1]}\otimes (p-2)^{[e+1]}\otimes 1^{[e+2]}$ and $H^2(G,V)=K$. This completes the proof for $p>2$.

Now assume that $p=2$. Then $r_0=0$. If $r_1=1$ then $E_2^{02}=E_2^{20}=0$ by \ref{23page} and so $H^2(G,V)=E_2^{11}=H^1(G,1\otimes r')$. Hence from \ref{lem} we get $r'=1\otimes 1^{[e+1]}$ with $e>0$. This gives $r=1^{[1]}\otimes 1^{[e+2]}$ and $H^2(G,V)=K$. If $r_1=0$ then by \ref{23page}, $E_2^{11}=0$. The same argument as above gives that if $E_2^{02}\neq 0$ then $E_2^{20}=0$, $d=2$, $r=s^{[d]}=1^{[2]}$ and $H^2(G,r)=K$. On the other hand if $E_2^{02}=0$ and $E_2^{20}\neq 0$, then $H^2(G,s^{[d]})=H^2(G,s^{[d-1]})$. As before, inductively we now see that $V$ is a Frobenius twist of $1^{[2]}$ or $1^{[1]}\otimes 1^{[e+2]}$ and $H^2(G,V)=K$.

{\it Proof of Corollary 2}

We briefly point out that one can combine \ref{ext1} and Theorem 1 to deduce Corollary 2. Let $V=L(s)^{[d]}$ with $d\geq 0$ and assume $p>2$ or $r\neq 2p$. Setting $r=0$ in \ref{ext1} the modules $V$ with $H^1(G,V)\neq 0$ are those with $s=p-2\otimes 1^{[1]}$. Now we may use \ref{ext1} again to see that the non-trivial modules which extend $V$ are 
\begin{align*}&2^{[d+1]} \text{ (if $k=d$ in \ref{ext1})},\\
&((p-2)\otimes (p-3)^{[1]}\otimes 1^{[2]})^{[d]}\text{ (if }k=d+1)\\
&((p-2)\otimes (p-2\pm1)^{[1]}\otimes 1^{[2]})^{[d-1]} \text{ (if }k=d-1),\\
&((p-2)^{[k]}\otimes 1^{[k+1]}\otimes (p-2)^{[d]}\otimes (p-2)^{[d+1]}),\text{ if }k\neq d,d\pm 1\end{align*}

Each irreducible $G$-module $V$ with $H^2(G,V)\neq 0$ appears above with the exception of $2^{[d]}=1^{[d=1]}$ when $p=2$ so the corollary follows.

{\footnotesize

\end{document}